\documentclass[11pt]{article}

\usepackage{amsfonts, amssymb}
\usepackage{amsmath}
\usepackage{color}
\usepackage{indentfirst}
\usepackage{verbatim}   
\usepackage{mathrsfs}   
\usepackage{yhmath}     

\newtheorem{theorem}{Theorem}[section]
\newtheorem{proposition}[theorem]{Proposition}
\newtheorem{lemma}[theorem]{Lemma}
\newtheorem{corollary}[theorem]{Corollary}
\newtheorem{prop-def}{Proposition-Definition}[section]

\newtheorem{defi}[theorem]{Definition}

\newtheorem{exam}[theorem]{Example}

\newenvironment{proof}{\trivlist \item[\hskip \labelsep{\it Proof.}]}{
 \endtrivlist}

\begin{document}

\title{Conjugacy Classes of Renner Monoids
\footnote{Project supported by national NSF of China (No 11171202).} }
\date{}
\author{Zhuo Li, Zhenheng Li, You'an Cao}
\maketitle

\vspace{ -1.2cm}

\vskip 7mm

\begin{abstract}
In this paper we describe conjugacy classes of a Renner monoid $R$ with unit group $W$, the Weyl group. We show that every element in $R$ is conjugate to an element $ue$ where $u\in W$ and $e$ is an idempotent in a cross section lattice. Denote by $W(e)$ and $W_*(e)$ the centralizer and stabilizer of $e\in \Lambda$ in $W$, respectively. Let $W(e)$ act by conjugation on the set of left cosets of $W_*(e)$ in $W$. We find that $ue$ and $ve$ ($u, v\in W$) are conjugate if and only if $uW_*(e)$ and $vW_*(e)$ are in the same orbit. As consequences, there is a one-to-one correspondence between the conjugacy classes of $R$ and the orbits of this action. We then obtain a formula for calculating the number of conjugacy classes of $R$, and describe in detail the conjugacy classes of the Renner monoid of some $\cal J$-irreducible monoids.

We then generalize the Munn conjugacy on a rook monoid to any Renner monoid and show that the Munn conjugacy coincides with the semigroup conjugacy, action conjugacy, and character conjugacy. We also show that the number of inequivalent irreducible representations of $R$ over an algebraically closed field of characteristic zero equals the number of the Munn conjugacy classes in $R$.

\vspace{ 0.3cm}
\noindent
{\bf Keywords:} Conjugacy, Renner Monoid, Weyl Group.

\vspace{ 0.3cm}
\noindent {\bf Mathematics Subject Classification 2010:}
20M32, 20G99.

\end{abstract}

\def\J {{\cal J}}
\def\a {\alpha}
\def\b {\beta}
\def\e {\varepsilon}
\def\s {\sigma}
\def\v {\vec}
\def\js {(\cal J, \sigma)}

\def\F {{\cal F}}
\def\ua {\underline a}
\def\ut {\underline t}
\def\e {\varepsilon}

\def\n { {\bf n} }
\def\r { {\bf r} }
\def\ek {\eta_{K}}
\def\eJ {\eta_{J}}
\newcommand{\be }{ \begin {}\\ \end{} }

\def\l {\langle}
\def\ra {\rangle}

\baselineskip 18pt

\section{Introduction}
There are different conjugacy relations in semigroup theory of which the following two are commonly studied ones. Let $R$ be a monoid with unit group $W$. Two elements $\sigma, \tau\in R$ are conjugate, denoted by $\sigma\sim \tau$, if there is $w\in W$ such that $\tau =  w \sigma w^{-1}$. Let $S$ be a semigroup. Then elements $\sigma, \tau\in S$ are called {\it primarily $S$-conjugate} if there are $x, y\in S$ for which $\sigma=xy$ and $\tau=yx$. This latter relation is reflexive and symmetric, but not transitive. Let $\equiv$ be its transitive closure, called semigroup conjugacy. If two elements in a monoid are $\sim$-conjugate, then they are semigroup conjugate. In a group these two conjugacy relations coincide and is equal to the usual group conjugacy.

The purpose of this paper is two-folded. We first describe the $\sim$-conjugacy classes in a Renner monoid $R$, and then investigate the connections between the semigroup conjugacy classes and irreducible representations of $R$. A Renner monoid is a finite inverse monoid induced from inductive (algebraic) monoids \cite{PU1, R1, R3, LS1}. It plays the same role for reductive monoids as the Weyl group does for reductive groups. The unit group of a Renner monoid $R$ is a Weyl group $W$. The symmetric inverse semigroup is a traditional example of Renner monoids; more examples can be found \cite{Zhenheng, LR}.

\subsection{The $\sim$-conjugacy classes}

The $\sim$-conjugacy in a monoid has been studied intensively with many useful results. It is well-known that two elements in a symmetric inverse semigroup $R_m$, called a rook monoid in combinatorics, are conjugate if and only if they have the same cycle-link types \cite{Lip}. A one-to-one correspondence between conjugacy classes in a symplectic rook monoid and (symplectic) partitions is established in \cite{CLL}, with precise formulas for calculating the number of conjugacy classes and formulas for computing the order of each class given. For a reductive monoid $M$, Putcha \cite{PU3} showed that there exist affine subsets $M_1, M_2, ..., M_k$ such that every element of $M$ is conjugate to an element of some $M_i$. Furthermore, he gave a necessary and sufficient condition for two elements in $M_i$ to be conjugate. Note that elements in different $M_i$ and $M_j$ may be conjugate. He then described precisely the conjugacy of elements in $M_i$ and $M_j$ with $i\ne j$ \cite{PU4}. Renner \cite{R2} investigated properties of conjugacy classes of semisimple elements in reductive monoids. Carter gave a complete description of conjugacy classes in the Weyl group \cite{C1, C2}; Humphreys summarized different developments of conjugacy for semisimple algebraic groups \cite{H2}.

In this paper, we first investigate the $\sim$-conjugacy classes of a Renner monoid $R$ using parabolic subgroups of the Weyl group $W$. To be specific, let $\Lambda$ be a cross section lattice of $M$. We obtain that every element in $R$ is conjugate to an element $ue$ for some $u\in W$ and $e\in \Lambda$. Let $W(e)$ and $W_*(e)$ be the centralizer and stabilizer of $e\in \Lambda$ in $W$, respectively (for definitions, see Section 2.3 below). Denote by $W/W_*(e)$ the set of left cosets of $W_*(e)$ in $W$ and let $W(e)$ act on $W/W_*(e)$ by conjugation. We show in Theorem 3.4 that if $u, v \in W$, then two elements $ue, ve\in R$ are conjugate if and only if $uW_*(e)$ and $vW_*(e)$ are in the same orbit. As a consequence, there is a one-to-one correspondence between the conjugacy classes and all the orbits (Corollary 3.5). This leads to a formula for calculating the number of conjugacy classes in $R$ (Corollary 3.6). As examples, we describe in detail the conjugacy classes in some $\cal J$-irreducible Renner monoids of type $A_2$, $B_2$, and $G_2$.

The study of $\sim$-conjugacy classes in Renner monoids reveals many important features of their structures and representations. One useful result in group representation theory is that the number of the usual group conjugacy classes in a finite group is equal to the number of irreducible representations of the group over an algebraically closed field with characteristic not a factor of the order of the group. However, this is not the case for Renner monoid representation theory. Indeed, there are usually more $\sim$-conjugacy classes than irreducible representations. For example, the rook monoid
$$
    R_2 = \Big\{\begin{pmatrix}
                0 & 0\\
                0 & 0\\
                \end{pmatrix},
                \begin{pmatrix}
                1 & 0\\
                0 & 0\\
                \end{pmatrix},
                \begin{pmatrix}
                0 & 1\\
                0 & 0\\
                \end{pmatrix},
                \begin{pmatrix}
                0 & 0\\
                1 & 0\\
                \end{pmatrix},
                \begin{pmatrix}
                0 & 0\\
                0 & 1\\
                \end{pmatrix},
                \begin{pmatrix}
                1 & 0\\
                0 & 1\\
                \end{pmatrix},
                \begin{pmatrix}
                0 & 1\\
                1 & 0\\
                \end{pmatrix}
         \Big\}
$$
has 5 $\sim$-conjugacy classes but 4 irreducible representations over an algebraically closed field of characteristic zero \cite{LLC, LS2}.

A natural question is how to define a new conjugacy relation different from the above in Renner monoids such that the number of corresponding conjugacy classes equals the number of irreducible representations?

\subsection{The semigroup conjugacy and an analogue of Munn conjugacy}

There are many elegant results about the semigroup conjugacy in the literature. Lallement studied it for free semigroups \cite{La}. Ganyushkin and Kormysheva showed that two elements in the symmetric inverse semigroup are conjugate if and only if they have the same stable rank and the restrictions to their stable images, respectively, have the same cycle type (see \cite{GK}). Ganyushkin and Mazorchuk described this conjugacy in detail in Chapter 6 of \cite{GM}. Kudryavtseva investigated the conjugacy in regular epigroups with elegant results \cite{Ku}. Kudryavtseva and Mazorchuk investigated semigroup conjugacy, action conjugacy, and character conjugacy for different semigroups in \cite{KM0}. They showed that the semigroup conjugacy and action conjugacy (see Section 2.1 for definitions) coincide for any inverse epigroups, and that the semigroup conjugacy and the character conjugacy are the same in regular epigroups with finite $\mathcal D$-classes. They studied conjugacy in Brauer-type semigroups and semigroups of square matrices in \cite{KM1, KM2}.

Munn \cite{M2} obtained the characters of irreducible representations of the symmetric inverse semigroup $R_m$, which implicitly introduced the so-called Munn conjugacy (see Definition \ref{conjugate}). Solomon \cite{LS2} made this conjugacy explicit, denoted by $\approx$. Two elements in $R_m$ are Munn conjugate if and only if they have the same cycle types \cite{GM, M2, LS2}. Thus the Munn conjugcay is equal to the semigroup conjugacy in $R_m$. Recently, the representation theory of Munn and Solomon has been generalized to Renner monoids \cite{LLC}, showing that the irreducible representations of a Renner monoid are completely determined by those of the parabolic subgroups of the Weyl group. Steinberg developed a representation theory of finite inverse semigroups with many useful applications \cite{S2}; his results are deeper and wider.

In this paper we find an analogue of the Munn conjugacy for a Renner monoid $R$ by embedding it into some symmetric inverse semigroup determined by the vertices of a polytope associated to $R$ (see Section 4.1 for more details). We then prove that the number of the Munn conjugacy classes in $R$ equals the number of the irreducible representations of $R$ over an algebraically closed field of characteristic zero (Theorem \ref{mainTheorem}).

What is the connection between the Munn conjugacy and the semigroup conjugacy in $R$? We show that these two conjugacy relations are equal. So the analogue of the Munn conjugacy provides a new description of the semigroup conjugacy for $R$ (Theorem \ref{munnEqualSconjugacy}). We go even further: in a Renner monoid the Munn conjugacy, character conjugacy, action conjugacy, and the semigroup conjugacy are all the same (Corollary \ref{allEqual}).

In the rest of the paper, Section 2 provides necessary facts and background information. Section 3 describes the $\sim$-conjugacy classes of Renner monoids $R$. Section 4 introduces an analogue of Munn conjugacy for $R$ and finds its connection with the representation theory of $R$.

\section{Preliminaries}

Let $S$ be a semigroup and $\sigma\in S$. Denote by $H_\sigma$ the $\mathcal H$-Green relation on $S$ (see \cite{CP, Ho, GM} for Green relations). An element $\sigma\in S$ is a group-bound element if there exists a positive integer $k$ such that $\sigma^k$ lies in a subgroup of $S$. If every element of $S$ is group-bound, we call $S$ an epigroup, which is also named as a group-bound semigroup or strongly $\pi$-regular semigroup in the literature. Every finite semigroup is an epigroup; so is the full matrix monoid consisting of all square matrices over a field. Let $\sigma\in S$ be group-bound such that $H_{\sigma^k}$ is a group whose identity element is denoted by $e_\sigma$. It follows from Lemma 1 of \cite{Ku} that the identity element $e_\sigma$ is well-defined. By Corollary 1 of \cite{Ku} we have $\sigma e_\sigma = e_\sigma \sigma$ and $\sigma e_\sigma \mathcal H e_\sigma$. The element $\sigma e_\sigma$ is called the {\it invertible part} of $\sigma$.

A semigroup is an inverse semigroup if every element has a unique inverse. Let $S$ be an inverse semigroup with the natural partial order on $S^1$ given by $\sigma\ge \tau$ if and only if there is an idempotent $e\in S$ such that $\tau=\sigma e$ (see also Chapter 5 of \cite {Ho}). An inverse semigroup with unit group $G$ is said to be factorizable if for each $\sigma\in S$ there is $g\in G$ such that $\sigma \le g$. The following result from (\cite{Ku}, Theorem 3) will be useful in Section 4.

\begin{theorem}\label{gandsconjugacy} Let $S$ be a factorizable inverse epigroup and $\sigma, \tau \in S$. Then $\sigma \equiv \tau$ if and only if $\sigma e_\sigma \sim \tau e_\tau$.
\end{theorem}

\subsection{Action conjugacy and character conjugacy}
We define action conjugacy and character conjugacy in a semigroup, and refer the reader to \cite{KM0} for more details. These two conjugacies will be used to compare with the Munn conjugacy in Section 4. Let $S$ be an inverse epigroup. Define a partial action of $S^1$ on $S$ by
\[
	\sigma \cdot x = \begin{cases} \sigma x\sigma^{-1}, \quad\quad \text{ if } \sigma^{-1}\sigma \ge e_x; \\
				  \text{undefined,} \quad \text{otherwise}.
                \end{cases}
\]

It follows from Lemma 1 in \cite{KM0} that if $\sigma, \tau \in S^1$ and $x\in S$ then $\tau\sigma\cdot x$ is defined if and only if
$\sigma\cdot x$ and $\tau\cdot (\sigma\cdot x)$ are both defined, in which $\tau\sigma\cdot x = \tau\cdot (\sigma\cdot x)$.

We call $x, y\in S$ {\it primary action conjugate} if there is $\sigma\in S^1$ for which $y=\sigma\cdot x$ or $x=\sigma\cdot y$.
This relation is reflexive and symmetric, but not necessarily transitive. Its transitive closure is called {\it
action conjugacy}.

Two elements $x, y$ in a semigroup $S$ are referred to as {\it character conjugate} if for every finite-dimensional
complex representation $\phi$ of $S$ we have $\chi_\phi(x) = \chi_\phi(y)$, where $\chi_\phi$ is the character
of $\phi$.

\subsection {Renner monoids}
Every linear algebraic monoid is an epigroup (\cite{PU1} Theorem 3.18). A linear algebraic monoid $M$ over an algebraic closed field is an affine algebraic variety together with an associative morphism from $M\times M$ to $M$ and an identity element $1\in M$. An irreducible algebraic monoid is a linear algebraic monoid whose underlying affine variety is irreducible; equivalently, $M$ is not the union of two proper closed nonempty subsets. The unit group of $M$, consisting of all invertible elements in $M$, is an algebraic group. An irreducible algebraic monoid is reductive if its unit group is a reductive group.

Let $M$ be a reductive algebraic monoid, $T\subseteq G$ a maximal torus
of the unit group $G$, $B\subseteq G$ a Borel subgroup with
$T\subseteq B$, $N$ the normalizer of $T$ in $G$, $\overline N$ the
Zariski closure of $N$ in $M$. Then $\overline N$ is a unit regular
inverse monoid which normalizes $T$, so $R=\overline N/T$ is a monoid and
        $$
        R=\overline N/T \supseteq N/T=W, \text { the Weyl group.}
        $$
\begin{defi} The monoid $R$ is called the Renner monoid of $M$.
\end{defi}

The Renner monoid $R$ is a finite factorizable inverse monoid, and hence an epigroup. The unit group of $R$ is a Weyl group. The idempotents in $R$ are exactly those in $\overline T$, the Zariski closure of $T$ in $M$. Moreover,
\[
    M = \bigsqcup_{r\in R} BrB, \quad \text{disjoin union}
\]
and if $s$ is a simple reflection then $BsB \cdot BrB \subseteq BsrB \cup BrB.$

\subsection{Cross section  lattice}
Denote by $E(\overline T) = \{e \in \overline T \mid e^2 = e\}$ the set of idempotents in $\overline T$. Partially order this set by defining
\[
    e \leq f \Leftrightarrow fe = e = ef.
\]
Then $E(\overline T)$ is a lattice with $e\wedge f = ef.$ The sublattice of $E(\overline T)$ given below
$$
        \Lambda = \{e\in E(\overline T)\mid Be=eBe\}
$$
is called the cross section lattice of $M$ and $R$. It is a useful concept, since
\[
    M = \bigsqcup\limits_{e \in \Lambda} GeG  \quad\text{and}\quad  R = \bigsqcup\limits_{e \in \Lambda} WeW,  \quad\mbox{disjoint unions.}
\]
\noindent
Furthermore, if $e, f \in E(\overline T)$ then $e \leq f \Leftrightarrow GeG \subseteq GfG \Leftrightarrow WeW \subseteq WfW$.

Let $W$ act naturally on $E(\overline T)$ by conjugation. The cross section lattice $\Lambda$ is a transversal of the orbits of this action. Each orbit is a conjugacy class of $E(\overline T)$ under $W$.

\subsection{Type map and $\cal J$-irreducible monoids}
The type map can be considered a monoid analogue of the Coxeter graph in Lie theory. It plays a crucial role in determining the cross section  lattice for a reductive monoid. Especially, for $\cal J$-irreducible monoids, it provides a combinatoric approach to precisely finding out the cross section  lattice. A reductive monoid $M$ with zero 0 is called $\J$-irreducible if $\Lambda\setminus\{0\}$ has a unique minimal element.

Let $V$ be a Euclidean space and let $r: W \rightarrow Gl(V)$ be the usual reflection representation of the Weyl group $W$.
Along with this goes the fundamental Weyl chamber $\mathcal C \subseteq V$ and the corresponding set of simple reflections
$S=\{s_\a \mid \a\in \Delta\}  \subseteq W$, where $\Delta$ is the set of simple roots of $G$ relative to $T$. Note that
$W$ is generated by $S$, and $\mathcal C$ is a fundamental domain for the action of $W$ on $V$. See \cite{H1, Ka} for details.

\begin{defi}  The type map
$
    \lambda: \Lambda\to 2^{\Delta}
$
is defined by
$
\lambda(e)=\{\a\in \Delta \mid s_\a e = es_\a\}.
$
\end{defi}
Let $\lambda^*(e) = \{\a\in \Delta\mid s_\a e=es_\a\not=e\}$ and $\lambda_*(e) = \{\a\in \Delta\mid s_\a e=es_\a =e\}$. Then  $\lambda(e) = \lambda^*(e) \bigsqcup \lambda_*(e)$.  We define parabolic subgroups of $W$ determined by $\lambda(e), \lambda^*(e)$ and $\lambda_*(e)$, respectively,
\[      W(e) = W_{\lambda(e)}, \qquad
        W^*(e) = W_{\lambda^*(e)},\qquad
        W_*(e) = W_{\lambda_*(e)}.
\]
Then $W(e) = \{w\in W \mid we = ew\}$ and $W_*(e) = \{w\in W \mid we = ew = e\}$. We call $W(e)$ the centralizer of $e$ in $W$ and $W_*(e)$ the stabilizer of $e$ in $W$. The following results are standard from Putcha \cite{PU1} and Renner \cite{R3}.

\begin{proposition}\label{idemlemma}
Let $e, f \in E(R)$ and $w \in W$.
\begin{enumerate}
\item If $we=f ~(\text{or } ew=f)$, then $e=f$,
\vspace {-2mm}
\item $we=e$ if and only if $ew=e$. Moreover, $W_*(e) = \{w\in W \mid we=e\}.$
\vspace {-2mm}
\item If $e\in \Lambda$, then $W_*(e)$ is a normal subgroup of $W(e) \cong W^*(e)\times W_*(e).$
\end{enumerate}
\end{proposition}

Let $G_0$ be a simple algebraic group and $\rho: G_0\to GL(V)$ be an irreducible rational representation over an algebraically closed field $K$. Then
\[
M=\overline {K^*\rho(G_0)}
\]
is a $\J$-irreducible monoid, called the $\J$-irreducible monoid associated with $\rho$. Let $\mu_i$ ($1\le i \le l$) be the fundamental dominant weights of $G_0$ of type $X$, where $X = $ $A_l$, $B_l$, $C_l$, $D_l$, $E_6$, $E_7$, $E_8$, $F_4$ and $G_2$.

\begin{defi}\label{BasicMonoid}
The $\J$-irreducible monoid associated with $\mu_i$ is called the $i$-th basic monoid of type $X$; its Renner monoid is referred to as the $i$-th basic Renner monoid of type $X$.
\end{defi}

The theorem below is a summary of Corollary 4.11 and Theorem 4.16 of \cite{PR2}, and we use the bracket notation $\langle \mu, \a\rangle$ as in page 16 of \cite{R3}.
\begin{theorem}\label{recipe}
Let $M$ be the $\J$-irreducible monoid associated with a highest weight representation with high weight a dominant weight $\mu$ and let $J_0=\{\a\in \Delta \mid \langle \mu,
\a\rangle=0\}$.
Then
\begin{enumerate}
\item $\lambda^*(\Lambda\setminus
\{0\})=\{X\subseteq\Delta\mid\text{X has no connected component
that lies entirely in } J_0 \}$.

\item $\lambda_*(e)=\{\a\in J_0\setminus \lambda^*(e) \mid
s_{\a}s_{\beta}=s_{\beta}s_{\a} \mbox{ for all } \beta \in
\lambda^*(e)\}, \mbox{ for } e\in \Lambda\setminus\{0\}$.
\end{enumerate}
\end{theorem}

\section{The $\sim$-conjugacy classes in Renner monoids}

Two elements $\sigma, \tau$ in a Renner monoid $R$ are conjugate, denoted by $\sigma\sim \tau$, if $\tau=w\sigma w^{-1}$ for some $w\in W$. Denote by  $W/W_*(e)$ the set of left cosets of $W_*(e)$ in $W$, and let
\[
D_*(e)=\{w\in W \mid l(ws_{\a})=l(w)+1 \mbox{ for all } \a\in
\lambda_*(e)\}.
\]
Then $D_*(e)$ is a set of left coset representatives of
$W/W_*(e)$, and each $w\in D_*(e)$ has a minimal length in $wW_*(e)$. If $\lambda_*(e) = \emptyset$, then $W_*(e) = 1$ and $D_*(e) = W$.

\begin{lemma} Each element in a Renner monoid $R$ is conjugate to
an element in $\{we \mid w\in D_*(e)\}$ for some $e\in \Lambda.$
\end{lemma}
\begin{proof}
    Let $r$ be an element in $R$. Then $r=uev$, where $u,v\in W$ and $e\in \Lambda$.
    So $r=v^{-1}vuev$. Thus, $r$ is conjugate to $vue$.  Let element $w\in D_*(e)$ be
    the left coset representative of the coset $vuW_*(e)$. Then $we = vue,$ and hence
    $r$ is conjugate to $we$.  $\hfill\Box$
\end{proof}

Let
$
    We = \{we \mid w\in W \}
$
for $e\in \Lambda$. The above lemma leads to the following corollary.

\begin{corollary} Each element in a Renner monoid is conjugate to
an element in $We$ for some $e\in \Lambda.$
\end{corollary}

\begin{lemma} No element of $We_1$ is conjugate to an element of $We_2$ for different idempotents $e_1, e_2\in \Lambda.$
\end{lemma}

\begin{proof} Assume that two elements $ue_1,ve_2 ~(u,v\in W \mbox{ and } e_1 \ne e_2)$ are conjugate. Then there exists an element $w\in W$ such that
 \[wue_1w^{-1}= ve_2.\]
Thus $(v^{-1}wuw^{-1})we_1w^{-1}=e_2$.  It follows from Proposition \ref{idemlemma} that $we_1w^{-1}=e_2$.  Since the cross-section lattice is a transversal of the conjugacy classes of $E(\overline T)$ under $W$, we have $e_1=e_2$, which contradicts the assumption. $\hfill\Box$
\end{proof}

Define a group action of $W(e)$ on $W/W_*(e)$ by
\[
    w\cdot uW_*(e) = wuw^{-1}W_*(e),
\]
where $w\in W(e)$ and $u\in W$. This action is well defined since $W_*(e)$ is a normal subgroup of $W(e)$.

\begin{theorem}\label{conjugaceTheoem} Let $e\in \Lambda$ and $W/W_*(e)$ be the set of left cosets of $W_*(e)$ in $W$. Two elements $ue, ve$ in $We$ are conjugate if and only if the two cosets $uW_*(e)$ and $vW_*(e)$ are in the same $W(e)$-orbit in $W/W_*(e)$.
\end{theorem}
\begin{proof}
    Let $u,v\in W$. If there exists $w\in W$ such that $wuew^{-1} =ve$, then
    \[
        (v^{-1}wuw^{-1})(wew^{-1}) =e.
    \]
It follows from Proposition \ref{idemlemma} that $wew^{-1}=e$ and $v^{-1}wuw^{-1}\in
W_*(e)$. Therefore, $w\in W(e)$ and $wuw^{-1} W_*(e) = v W_*(e)$, that is, $uW_*(e)$ and $vW_*(e)$ are in the same $W(e)$-orbit in $W/W_*(e).$ The argument can be reversed.  $\hfill\Box$
\end{proof}

The following results are corollaries of Theorem \ref{conjugaceTheoem}.

\begin{corollary}
There is a one-to-one correspondence between the conjugacy classes of a Renner monoid and the orbits of the above action of $W(e)$ on $W/W_*(e)$ where $e$ runs through $\Lambda$.
\end{corollary}

\begin{corollary} Let $n_e$ be the number of $W(e)$-orbits in $W/W_*(e)$. Then the number of the conjugacy classes in a Renner
monoid is
$
    \sum_{e\in \Lambda} n_e.
$
\end{corollary}

\noindent{\it{\bf Example 3.7}}
Let $M$ be the $\cal J$-irreducible algebraic monoids associated with an irreducible representation $\rho$ of a simple algebraic group of type $A_2$, $B_2$ and $G_2$, respectively,
such that $\langle \rho, \a_1 \rangle \ne 0$ and $\langle \rho, \a_2 \rangle \ne 0$. Then $M$ is a canonical monoid with $J_0 = \{\a\in \Delta \mid \langle \rho, \a \rangle = 0 \} = \emptyset$. Note that $\Delta = \{\a_1, \a_2\}$ and $S = \{s_{\a_1}, s_{\a_2}\}$. Write $s_1=s_{\a_1}$ and $s_2=s_{\a_2}$, once and forever. By Theorem \ref{recipe},
\[
    \Lambda\setminus\{0\} \cong \{\emptyset, ~\{\a_1\}, ~\{\a_2\}, ~\{\a_1, \a_2\}\}.
\]
So $\Lambda = \{0, ~e_0, ~e_1, ~e_2, ~1\}$ with $e_0$ the minimal nonzero idempotent. The conjugacy classes of the Renner monoid of $M$ is summarized as follows.

\vspace{1mm}
If $X = A_2,$ then $R$ has 18 conjugacy classes with representatives
\[
\begin{aligned}
\{ &~0,  ~e_0, ~s_1e_0, ~s_2e_0, ~s_1s_2e_0, ~s_2s_1e_0, ~s_1s_2s_1e_0, ~e_1, \\
   &~s_1e_1, ~s_2e_1, ~s_1s_2e_1, ~e_2, ~s_1e_2, ~s_2e_2, ~s_1s_2e_2, ~1, ~s_1, ~s_1s_2 ~\}.
\end{aligned}
\]

\vspace{1mm}
If $X = B_2,$ then $R$ has 26 conjugacy classes with representatives
\[
\begin{aligned}
\{ &~0,  ~e_0, ~s_1e_0, ~s_2e_0, ~s_1s_2e_0, ~s_2s_1e_0, ~s_1s_2s_1e_0, ~s_2s_1s_2e_0, ~-e_0, \\
   &~e_1, ~s_1e_1, ~s_2e_1, ~s_1s_2e_1, ~s_2s_1s_2e_1, ~-e_1, ~e_2, ~s_1e_2, ~s_2e_2, \\
   &~s_1s_2e_2, ~s_1s_2s_1e_2, ~-e_2, ~1, ~s_1, ~s_2,~s_1s_2, ~-1 ~\}.
\end{aligned}
\]

\vspace{1mm}
If $X = G_2,$ then $R$ has 35 conjugacy classes with representatives
\[
\begin{aligned}
\{&~0, ~e_0, ~s_1e_0, ~s_2e_0, ~s_1s_2e_0, ~(s_1s_2)^2e_0,  ~-e_0, ~(s_2s_1)^2e_0, ~s_2s_1e_0, ~s_2s_1s_2e_0, \\
&~(s_2s_1)^2s_2e_0, ~(s_1s_2)^2s_1e_0, ~s_1s_2s_1e_0, ~e_1, ~s_1e_1, ~s_2e_1, ~s_1s_2e_1, ~(s_1s_2)^2e_1,  \\
&~-e_1, ~s_2s_1s_2e_1, ~-s_1e_1, ~e_2, ~s_1e_2, ~s_2e_2, ~s_1s_2e_2, ~(s_1s_2)^2e_2, ~-e_2, \\
&~s_1s_2s_1e_2, ~-s_2e_2,  ~1, ~s_1, ~s_2, ~s_2s_1, ~(s_2s_1)^2, ~-1 ~\}.
\end{aligned}
\]

We show a detailed calculation for case $G_2$; the other two cases are similar. The Weyl group $W$ of type $G_2$ is isomorphic to the dihedral group of order 12 generated by two elements $s_1s_2$ and $s_2$, where $s_1$ and $s_2$ are the simple reflections such that $(s_1s_2)^6 = s_2^2 = 1$ and $s_2(s_1s_2)s_2^{-1} = (s_1s_2)^{-1}$. The elements of $W$ are $(s_1s_2)^k, ~(s_1s_2)^ks_2,$ for $0\le k \le 5,$ and we have the relation $s_2(s_1s_2)^k = (s_1s_2)^{-k}s_2$ and $(s_1s_2)^3 = -1$. Thus
\[
    W = \{1, ~s_1s_2, ~(s_1s_2)^2, ~-1, ~(s_1s_2)^4, ~(s_1s_2)^5, ~s_2, s_1, ~(s_1s_2)s_1, ~(s_1s_2)^2s_1, ~-s_1, ~s_2s_1s_2 \}.
\]

If $e=0$, it determines the conjugacy class consisting of $0$ only.

If $e=e_0$, by Theorem \ref{conjugaceTheoem} $D_*(e_0)$ contributes 12 classes with representatives $ue$ for $u\in W$. Indeed, $W(e) = W_*(e) = 1$ and $D_*(e) = W$. Each element of $D_*(e)$ gives rise to a conjugacy class containing a single element.

If $e=e_1$, it follows from Theorem \ref{conjugaceTheoem} that $D_*(e_1)$ gives rise to 8 conjugacy classes with representatives:
$e_1$, $~s_1e_1$, $~s_2e_1$, $~s_1s_2e_1$, $~(s_1s_2)^2e_1$, $-e_1$, $~s_2s_1s_2e_1$, $~-s_1e_1.$ In fact, in this case, $W(e) = \{1, s_1\}$, $W_*(e) = \emptyset$, and $D_*(e) = W \cong W/W_*(e)$. Notice that each of $1, ~s_1, (s_1s_2)^3, -s_1$ is in an orbit consisting of a single element. Also, $s_1s_2$ and $s_2s_1$ are in one orbit; $(s_1s_2)^2$ and $(s_1s_2)^4=(s_2s_1)^2$ are in a orbit; $s_2$ and $s_1s_2s_1$ are in the same orbit; $s_2s_1s_2$ and $(s_1s_2)^2s_1$ are in one orbit.

If $e=e_2$, then by symmetry of the Dynkin diagram of $G_2$ we see that $D_*(e_2)$ provides 8 conjugacy classes with representatives:
$e_2$, $~s_1e_2$, $~s_2e_2$, $~s_1s_2e_2$, $~(s_1s_2)^2e_2$, $-e_2$, $~s_1s_2s_1e_2$, $~-s_2e_2$.

If $e=1$, then $W_*(e) = 1$ and $D_*(e) = W(e) = W$, which is isomorphic to the dihedral group with 12 elements. So $W$ has 6 conjugacy classes (c.f. \cite{J}, Chapter 5) with representatives:
$1,$ $s_1,$ $s_2,$ $s_2s_1,$ $(s_2s_1)^2,$ $-1$. \hfill $\Box$

\vspace{2mm}
\noindent{\it{\bf Example 3.8}} We now let $R$ be the first basic Renner monoid of type $X$, where $X = A_2, B_2$ or $G_2$. Then $\Delta = \{\a_1, \a_2\}$ and $S = \{s_1, s_2\}$. By Theorem \ref{recipe},
\[
    \Lambda\setminus\{0\} = \{\emptyset, ~\{\a_1\}, ~\{\a_1, \a_2\}\}.
\]
So $\Lambda = \{0, ~e_0, ~e_1, ~1\}$ with $e_0$ the minimal nonzero idempotent. We list the following results without showing the details.

If $X = A_2,$ then $R$ has 10 conjugacy classes with representatives

\centerline{
$
\{~0,  ~e_0, ~s_1e_0, ~e_1, ~s_1e_1, ~s_2e_1, ~s_1s_2e_1, ~1, ~s_1, ~s_1s_2
~\}.
$
}

\vspace{1mm}
If $X = B_2,$ then $R$ has 15 conjugacy classes with representatives

$
\{~0, ~e_0, ~s_1e_0, ~-e_0, ~e_1, ~s_1e_1, ~s_2e_1, ~s_1s_2e_1, ~s_2s_1s_2e_1, ~-e_1, ~1, ~s_1, ~s_2,~s_1s_2, ~-1~\}.
$

\vspace{1mm}
If $X = G_2,$ then $R$ has 19 conjugacy classes with representatives
\[
\begin{aligned}
\{&~0, ~e_0, ~s_1e_0, ~-s_1e_0, ~-s_2e_0,  ~e_1, ~s_1e_1, ~s_2e_1, ~s_1s_2e_1, ~(s_1s_2)^2e_1, \\
~&~-e_1, ~s_2s_1s_2e_1, ~-s_1e_1, ~1, ~s_1, ~s_2, ~s_1s_2, ~(s_1s_2)^2, ~-1 ~\}.
\end{aligned}
\]


\section{Munn conjugacy and its connections to representations}
We plan to generalize the Munn conjugacy on rook monoids $R_m$ to any Renner monoid $R$ and then investigate the connection of this conjugacy with representations of $R$.

\subsection{Embedding a Renner monoid into a rook monoid}
Closely connected with a reductive monoid $M$ is a polytope $P$ (\cite{PU1, R3}). For simplicity, we call $P$ the
polytope associated with the Renner monoid $R$ of $M$, if there is no confusion. Denote by $V(P) = \{1, ..., m\}$
the set of all vertices of $P$. A subset $K \subseteq V(P)$ is regarded as a face of the polytope if it is the set
of vertices of a face of $P$. Let $\F(P)$ be the face lattice of $P$. For any $K \in \F(P)$, define
\begin{equation}\label{edef}
    e_K (i) = \left\{ \begin{array}{ll}
                        i                       &\text{\quad if  $i\in K,$}\\
                        \text{undefined}        &\text{\quad if  $i\notin K$}.
                       \end{array} \right.
\end{equation}
By convention, $e_\emptyset = 0$, and $e_{V(P)}$ is the identity map. It follows from \cite{PU1} that
$$
    E(R) = \{e_K \mid K\in \F(P)\}.
$$

The conjugation action of $W$ on $E(R)$ induces an action of $W$ on $\F(P)$ as follows. For $w\in W$ and $J, K \in \F(P)$,
\begin{eqnarray}\label{action}
    wJ = K, \quad \mbox{ if \quad } we_Jw^{-1}  = e_K.
\end{eqnarray}
If $e\in E(R)$ and $e=e_L$, let $\mathcal F(e) = \{wL \mid w \in W\}$ be the orbit of $L$ under $W$.

Lemma 3.2 of \cite{LLC} shows that, for any $\sigma \in R$, there exist $w, w_1 \in W$ and unique faces $I, J$ of $P$ such that $\sigma = e_Jw = w_1e_I$. For any $i\in I$, let $\sigma(i)$ denote the image of $i$ under $\sigma$.  The product of $\tau, \sigma \in R$ is regarded as $\tau\sigma(i) = \tau(\sigma(i))$ if $i\in I(\sigma)$ and $\sigma(i)\in I(\tau)$. Then
\begin{eqnarray}\label{sigmaJK}
    \sigma = e_J w e_I
\end{eqnarray}
is a map of $I$ onto $J$. Thus $R$ is a submonoid of $R_m$.

\begin{defi}
 The face $I$ is called the domain of $\sigma$ and will be denoted by $I(\sigma)$; the face $J$ is called the range of $\sigma$ and will be denoted by $J(\sigma)$.
\end{defi}
\noindent Clearly if $w\in W$ then $I(w) = J(w) = V(P)$.

\subsection{Munn conjugacy}
To obtain the desired result that the number of conjugacy classes of any Renner monoid $R$ equals the number of its inequivalent irreducible representations, we develop an analogue of the Munn conjugacy relation for Renner monoids. For $\sigma\in R$, let $I^\circ(\sigma)$ be the set of $i\in I(\sigma)$ such that $\sigma^k(i)$ is defined for all $k\ge 1$, that is,
\[
    I^\circ(\sigma) = \bigcap_{k = 0}^{\infty}{I(\sigma^k)}.
\]
Then $I^\circ(\sigma) \subseteq I(\sigma)$ and $\sigma$ fixes $I^\circ(\sigma)$. Let
$$
    \sigma^\circ = \sigma|_{_{I^\circ(\sigma)}},
$$
the restriction of $\sigma$ on $I^\circ(\sigma)$. For example, if the Renner monoid is the rook monoid $R_6$ and $\sigma: 1 \mapsto 5 \mapsto 6 \mapsto 1$ and $2\mapsto 4$ with $\sigma(3), \sigma(4)$ undefined, then $I^\circ(\sigma) = \{1, 5, 6\}$ is a subset of the domain $\{1, 2, 5, 6\}$ of $\sigma$, and $\sigma^\circ : 1\mapsto 5 \mapsto 6 \mapsto 1$.

\begin{lemma}
Let $\sigma\in R$. Then $\sigma^\circ$ is an element of $R$.
\end{lemma}
\begin{proof}
It follows from the definition of $\sigma^\circ$ that the range $J(\sigma^\circ)$ of $\sigma^\circ$ is the same as the range of $\sigma^k$ for some $k\ge 1.$ Thus $J(\sigma^\circ)$ is a face of $P$, since the range of $\sigma^k\in R$ is a face of $P$. Notice that the domain $I(\sigma^\circ)$ of $\sigma^\circ$ is equal to $J(\sigma^\circ)$. Then $e_{_{I{(\sigma^\circ)}}} \in R.$ Let $\sigma = w e_I$ for some $w\in W$. Then $ \sigma^\circ = \sigma|_{_{I(\sigma^\circ)}} = w e_{_{I{(\sigma^\circ)}}} \in R.$ \hfill $\Box$
\end{proof}

\begin{defi}\label{conjugate} Let $W$ be the unit group of $R$. Then two elements $\sigma, \tau \in R$ are called Munn {\em conjugate}, denoted by $\sigma \approx \tau$, if there exists $w\in W$ such that $w^{-1}\sigma^\circ w = \tau^\circ$. The Munn conjugacy class of $\sigma$ is denoted by $[\sigma]$.
\end{defi}

Two elements in the rook monoid $R_m$ are Munn conjugate if and only if they have the same cycle pattern \cite{M2}. There are 7 Munn conjugacy classes in $R_3$ with representatives:
\noindent
$
0, ~(1)[2][3], ~(1)(2)[3], ~(12)[3], ~(12)(3), ~(123), ~1.
$

\begin{theorem}\label{munnEqualSconjugacy} The Munn conjugacy and the semigroup conjugacy coincide in a Renner monoid $R$.
\end{theorem}

\begin{proof}
Let $\sigma = w e_{I(\sigma)}\in R$, where $w\in W$ and $I(\sigma)$ is the domain of $\sigma$. Write $I = I(\sigma)$. Then
we have
\[
    \sigma^2 = we_I we_I= w^2 e_{w^{-1}(I)} e_I = w^2 e_{I \cap w^{-1}(I)}.
\]
Let $k$ be the smallest positive integer for which $w^k(I) = I$. Then
\[
    I^\circ(\sigma) = I\cap w(I)\dots \cap w^{k-1}(I).
\]
Thus $\sigma^k = w^k e_{I^\circ(\sigma)} = e_{I^\circ(\sigma)} w^k \in H_{\sigma^k}$ with the identity element $e_{I^\circ(\sigma)}$. So the invertible part of $\sigma$
is $\sigma e_{I^\circ(\sigma)} = w e_{I(\sigma)} e_{I^\circ(\sigma)} = w e_{I^\circ(\sigma)} = \sigma^\circ$. It follows from
Theorem \ref{gandsconjugacy} and Definition \ref{conjugate} that the Munn conjugacy coincides with the semigroup conjugacy
in $R$.
$\hfill\Box$
\end{proof}

We would like to point out that the result in Theorem 4.4 can also be deduced from results in \cite{KM0}. In fact, we can prove directly that the Munn conjugation coincides with the action conjugation in $R$. Indeed, if $\sigma, \tau\in R$ are Munn conjugate, then there is $w\in W$ such that $w\cdot \sigma^\circ = w \sigma^\circ w^{-1} = \tau^\circ.$ But $e_\sigma \cdot \sigma = e_\sigma \sigma e_\sigma  = \sigma e_\sigma =\sigma^\circ$. It follows that $(we_\sigma)\cdot \sigma = \tau^\circ$. Notice that $e_\tau \cdot \tau = \tau^\circ$. Thus $\sigma$ and $\tau$ are action conjugate. On the other hand, if $\sigma$ and $\tau$ are primary action conjugate, then (without loss of generality) there is $\alpha\in R$ such that $\alpha^{-1}\alpha \ge e_\sigma$ and $\alpha\cdot \sigma = \tau$.  Following the line of \cite{KM0}, we obtain that $\sigma^\circ$ and $\tau^\circ$ are semigroup conjugate. Thanks to Corollary 6 in \cite{Ku} and Theorem \ref{gandsconjugacy}, we see that $\sigma$ and $\tau$ are Munn conjugate. It follows easily that if $\sigma$ and $\tau$ are action conjugate, then they are Munn conjugate, since action conjugacy is the transitive closure of the primary action conjugacy and the Munn conjugacy is an equivalency relation.

\begin{corollary}\label{allEqual}
The action conjuate, character conjuate, Munn conjugate, and semigroup conjugate are all the same in a Renner monoid.
\end{corollary}

\begin{proof}
The result follows from Corollary 5 of \cite{KM0}, since the Renner monoid is a finite inverse monoid.
\end{proof}

\begin{corollary}
Two idempotents are in the same Munn conjugacy class if and only if they are $\sim$-conjugate with each other.
\end{corollary}

Note that Theorem \ref{munnEqualSconjugacy} also provides an algorithm to calculate the Munn conjugacy classes in R.
We give an example below.

\begin{exam}
{\rm Let $R$ be the first basic Renner monoid of type $A_2$. Then $R$ is isomorphic to the rook monoid $R_3$, and its unit group is the Weyl group generated by two simple
reflections $s_1$ and $s_2$. The cross section  lattice of $R$ is $\Lambda = \{0, ~e_0, ~e_1, ~1\}$ with $e_0$ the minimal
nonzero idempotent. By Example 3.8, there are 10 $\sim$-conjugacy classes in $R$, with representatives

\centerline{
$
\{~0,  ~e_0, ~s_1e_0, ~e_1, ~s_1e_1, ~s_2e_1, ~s_1s_2e_1, ~1, ~s_1, ~s_1s_2
~\}.
$
}
\noindent
Their corresponding invertible parts are

\centerline{
$
\{~0,  ~e_0, ~0, ~e_1, ~s_1e_1, ~e_0, ~0, ~1, ~s_1, ~s_1s_2 ~\},
$
}
\noindent
and hence $R$ has 7 Munn conjugacy classes

\centerline{
$
\{~0,  ~e_0, ~e_1, ~s_1e_1, ~1, ~s_1, ~s_1s_2 ~\}.
$
}

Similarly, since the first basic Renner monoid of type $B_2$ has 15 $\sim$-conjugacy classes

$
\{~0, ~e_0, ~s_1e_0, ~-e_0, ~e_1, ~s_1e_1, ~s_2e_1, ~s_1s_2e_1, ~s_2s_1s_2e_1, ~-e_1, ~1, ~s_1, ~s_2,~s_1s_2, ~-1~\},
$

\noindent
a simple calculation yields that the Renner monoid has 9 Munn conjugacy classes

\centerline{
$
\{~0, ~e_0, ~e_1, ~s_1e_1, ~1, ~s_1, ~s_2,~s_1s_2, ~-1~\}.
$
}

}

\end{exam}

\subsection{Projections of $R$ into its parabolic subgroups}
Directly generalizing Solomon \cite{LS2}, we now define a projection from $R$ to $W^*(e)$, where $e\in \Lambda$.
Recall that we use $P$ to denote the ploytope associated with $E(R)$.

\begin{proposition}\label{projectionMapping} Let $K, L$ be two faces of $P$. If $K$ and $L$ are in the same
$W$-orbit and the relative interior of $L$ intersects the closure of the fundamental Weyl chamber, then there
exists a unique element $w\in W$ with shortest length such that $w(L)=K$.
\end{proposition}
\begin{proof}
It suffices to prove the uniqueness of such element with shortest length.  If $K=L$, then $w=1$. Now, suppose
that $K \ne L$. Let $W_L$ be the parabolic subgroup of $W$ that fixes $L$. Since the relative interior of $L$ intersects the closure of the fundamental
Weyl chamber, $W_L$ is a standard parabolic subgroup.  Let $W^L = \{u\in W \mid l(uv) = l(u) + l(v) \text{ for all } v\in W_L\}$. Each element $w\in W$ has a unique decomposition $w=uv$
where $u\in W^L$ and $v\in W_L$, and
\begin{eqnarray}\label{length}
l(w)=l(u)+l(v).
\end{eqnarray}
Assume $w(L) = w_1(L)=K$ for $w, w_1\in W$ such that $w$ and $w_1$ with shortest length. Then
$w, w_1\in W^L$ and $w_1^{-1}w\in W_L$. Hence, $w=w_1 x$ for some $x \in W_L$, and $l(w)=l(w_1)+l(x)$
by (\ref{length}). Thus $w=w_1$, since $l(w)=l(w_1)$ forces $x=1$. $\hfill\Box$

\end{proof}

Let $e\in \Lambda$ and $L$ be the unique face of $P$ for which $e = e_L$. Then the intersection of the relative
interior of $L$ and the closure of the fundamental Weyl chamber is not empty. It follows from Proposition
\ref{projectionMapping} that for any face $K$ in the orbit of $L$ under the action (\ref{action}) there is
$w\in W$ with the shortest length such that $wL = K$ and $w e_{L} w^{-1} = e_K$. Hence
$$
    \mu_K = w e_L
$$
is an element of $R$ and maps $L$ to $K$. Similarly, $\mu_K^-= w^{-1} e_K \in R$ is a map of $K$ onto $L$ and \begin{eqnarray}\label{kk}
    \mu_K\mu_K^- = w e_{L} w^{-1} e_K = e_K.
\end{eqnarray}

We now introduce the projection from $R = \bigsqcup_{e\in \Lambda} WeW$ to $W^*(e)$ for every $e\in \Lambda$. If $\sigma = e_J u e_I\in WeW$, define
\begin{eqnarray}\label{psigma}
    p(\sigma) = \mu_J^-\sigma\mu_I.
\end{eqnarray}
Then $p(\sigma)$ maps ${L}$ to ${L}$, and $p(\sigma) e = e p(\sigma)$ by (ii) of Lemma 3.3 in \cite{LLC}, that is, $p(\sigma) \in W(e)$. But then $p(\sigma) = e p(\sigma)$ forces  $p(\sigma)\in eW(e)$. However, $eW(e) = W^*(e)$; we have
$
    p(\sigma) \in W^*(e).
$
It follows from (\ref{kk}) that
$
    \sigma = \mu_J\mu_J^-\sigma\mu_I\mu_I^-  = \mu_J p(\sigma) \mu_I^-.
$
If $\sigma\in W$, then $p(\sigma) = \sigma \in W$.

 \begin{lemma}\label{projection}
 If $\sigma, \tau, \tau\sigma \in WeW$ and $J(\sigma) = I(\tau)$, then $p(\tau\sigma) = p(\tau)p(\sigma)$.
 \end{lemma}
 \begin{proof}
 Write $I = I(\sigma), J = J(\sigma)$, and $K = J(\tau)$. Then $\tau=e_K w_1 e_J$ and $\sigma = e_J w e_I$ for some $w_1, w \in W$. It follows that $\tau e_J\sigma = e_Kw_1 e_J w e_I =\tau\sigma$. So $p(\tau)p(\sigma) = \mu_K^- \tau \mu_J \mu_J^- \sigma \mu_I = \mu_K^- \tau e_J \sigma \mu_I = \mu_K^- \tau\sigma \mu_I = p(\tau\sigma)$.
 \end{proof}

\subsection{Connections to Representations}

\begin{defi}\label{subrank}
Let $\sigma\in R$. If $\sigma^\circ\in WeW$ for some $e\in \Lambda$, then $e$ is referred to as the subrank of $\sigma$.
\end{defi}

\begin{lemma}\label{basicLemma} Let $\sigma, \tau\in R$.

{\upshape{(i)}} If $\sigma \sim \tau$ then $\sigma \approx \tau$.

{\upshape{(ii)}} All elements in $[\sigma]$ have the same subrank.

{\upshape{(iii)}} If $e = e_L\in\Lambda$ for some $L\in \F(P)$ and $\sigma\in W^*(e)$, then $\sigma^\circ = \sigma$ and $I^\circ(\sigma) = L$.

\end{lemma}
\begin{proof}
For (i), since $\sigma$ and $\tau$ are conjugate, we have $w^{-1}\sigma w = \tau$ for some $w\in W$. Hence $w^{-1}\sigma^\circ w = \tau^\circ$, in other words, $\sigma \approx \tau$. Result (ii) is straightforward from Definitions \ref{conjugate} and \ref{subrank}. As for (iii), notice that $\sigma\in eW(e) = W(e)e$. We obtain that $\sigma = e\sigma e = e_L\sigma e_L$. Thus (iii) follows.  $\hfill\Box$
\end{proof}

\begin{lemma} \label{meetW}
Let $\sigma\in R$ with subrank $e\in \Lambda$. Then $[\sigma]$ meets one and only one parabolic subgroup of the form
$\{W^*(f) \mid f \in \Lambda\}$. Specifically, $[\sigma]$ meets $W^*(e)$.
\end{lemma}
\begin{proof} Let $I$ be the domain of $\sigma^\circ$. Then $\sigma^\circ = w_1 e_I \in WeW$ for some $w_1\in W$.
Let $L$ be the unique face of $P$ such that $e = e_L$. It follows from (ii) of Lemma 3.3 in \cite{LLC} that $I$ is
in the orbit of $L$ under the action (\ref{action}), and so $wL = I$ for some $w\in W$. Thus
$
    p(\sigma^\circ) = \mu_I^- \sigma^\circ \mu_I = (e_Lw^{-1}) \sigma^\circ (we_L).
$
We have
\begin{align}\label{psigmaquan}
    p(\sigma^\circ) = w^{-1} \sigma^\circ w.
\end{align}
But $p(\sigma^\circ)^\circ = p(\sigma^\circ)$, since $p(\sigma^\circ)\in W^*(e)$. We see that $p(\sigma^\circ) \in [\sigma] \bigcap W^*(e)$. The uniqueness of
$W^*(e)$ follows from the fact that all the elements of $[\sigma]$ have the same subrank $e\in \Lambda.$
$\hfill\Box$
\end{proof}

The equality (\ref{psigmaquan}) implies the result below.
\begin{corollary}\label{sigmaPsigma} If $\sigma\in R$, then $[\sigma] = [p(\sigma^\circ)].$
\end{corollary}

\begin{lemma}\label{conjugacyClass} If $\sigma\in R$ has subrank $e\in \Lambda$, then
$[\sigma] \cap W^*(e) = \overline{p(\sigma^\circ)}$, the $\sim$-conjugate class of $p(\sigma^\circ)$ in $W^*(e).$
\end{lemma}

\begin{proof}
It suffices to show that if $x \in [\sigma] \cap W^*(e)$, then $x$ is conjugate to $p(\sigma^\circ)$ in $W^*(e).$ Clearly $\sigma^\circ = u^{-1} x^\circ u$ for some $u\in W$, since $x \in [\sigma]$. Now $p(\sigma^\circ) = w^{-1} \sigma^\circ w$ for some $w\in W$ by (\ref{psigmaquan}). However, $x^\circ = x$ due to $x\in W^*(e)$. Thus
\begin{equation*}
    p(\sigma^\circ) = ep(\sigma^\circ)e = ew^{-1}u^{-1} x uwe.
\end{equation*}
Note that $ew^{-1}u^{-1}, x, uwe \in WeW$, $J(uwe) = I(x)$, and $J(x)=I(ew^{-1}u^{-1})$. In view of Lemma \ref{projection}, we have
\begin{align}
     p(\sigma^\circ) &= p(ew^{-1}u^{-1}) p(x) p(uwe) \notag \\
                     &= p(uwe)^{-1} \,x\, p(uwe),
\end{align}
where $p(uwe)\in W^*(e)$. It follows that $[\sigma] \cap W^*(e) \subseteq \overline{p(\sigma^\circ)}$. The reverse inclusion is easily seen. $\hfill\Box$

\end{proof}

\begin{corollary} Let $\sigma, \tau \in W^*(e)$, where $e\in \Lambda$. Then $\sigma\sim \tau$ in $W^*(e)$ if and only if $\sigma\approx \tau$ in $R$.
\end{corollary}
\begin{proof}
It is straightforward from Lemma \ref{conjugacyClass}. $\hfill\Box$
\end{proof}

\begin{theorem}\label{numOfClasses}
The number of Munn conjugacy classes of a Renner monoid $R$ equals the total number of all $\sim$-conjugate classes of $W^*(e)$ for all $e\in\Lambda$.
\end{theorem}

\begin{proof} It follows from Lemma \ref{conjugacyClass} that the Munn conjugate classes of $R$ that meet $W^*(e)$
are indexed by conjugate classes of $W^*(e)$ for $e\in \Lambda$. Thanks to Lemma \ref{meetW} and $R = \bigsqcup_{e\in \Lambda} WeW$, the desired result follows. $\hfill\Box$
\end{proof}

\begin{theorem} \label{mainTheorem}
The number of inequivalent irreducible representations of a Renner monoid $R$ over an algebraically closed field of characteristic zero equals the total number of Munn conjugacy classes in $R$.
\end{theorem}

\begin{proof}
Theorem 3.2 of \cite{LLC} shows that the full set of inequivalent irreducible representations of $R$ over a field of characteristic zero is completely determined by a full set of inequivalent irreducible representations of $W^*(e)$ for all $e\in\Lambda$. On the other hand, from group representation theory, the number of irreducible representations of $W^*(e)$ is the same as the number of the conjugate classes of $W^*(e)$. The result we want follows from Theorem \ref{numOfClasses}. $\hfill\Box$
\end{proof}


\begin{corollary} Let $p(r)$ be the number of partitions of $r$ for $0\le r \le m.$ Then the number of Munn conjugacy classes of the symmetric inverse semigroup $R_m$ is
\[
    \sum_{r = 0}^m p(r).
\]
\end{corollary}
\begin{proof}
It follows form \cite{LS2} Theorem 2.24 that there is a one-to-one correspondence between the set of inequivalent irreducible representations of $R_m$ and the set of all the partitions of $r$ with $0\le r \le m.$ The desired result follows from Theorem \ref{mainTheorem}.
$\hfill\Box$
\end{proof}

In a sequel paper we will investigate the relationship between conjugacy classes of a finite inverse monoid
and its representations.

\vspace{3mm}
{\bf Acknowledgments}
We would like to thank the referee for bringing \cite{KM0} to our attention and for many valuable suggestions, which make the paper deeper and more readable. The authors also thank Dr. Reginald Koo for his useful comments.

\vspace{3mm}
\noindent Zhuo Li\\
Department of Mathematics \\
Xiangtan University\\
Xiangtan, Hunan 411105, China\\
\noindent Email: zli@xtu.edu.cn\\

\noindent Zhenheng Li\\
Department of Mathematical Sciences \\
University of South Carolina Aiken\\
Aiken, SC 29801, USA\\
\noindent Email: zhenhengl@usca.edu\\

\noindent You'an Cao\\
Department of Mathematics \\
Xiangtan University\\
Xiangtan, Hunan 411105, China\\
\noindent Email: cya@xtu.edu.cn\\


\begin{thebibliography}{1}

\bibitem{CLL} Y. Cao, Z. Li and Z. Li, Conjugacy classes in the symplectic Renner monoid,   Journal of Algebra, 324(8) 2010, 1940-1951.

\bibitem{C1} R. Carter, Conjugacy classes in the Weyl group, Compositio Mathematica, 25(1) 1972, 1-59.
\bibitem{C2} R. Carter, Finite groups of Lie type: Conjugacy classes and complex characters, Wiley, 1985.

\bibitem{CP} A. Cliford and G. Preston, The algebraic theory of semigroups, Vol I, Mathematical Surveys, No. 7. American Mathematical Society, Providence, RI, 1961.


\bibitem{GK} O. Ganyushkin and T. Kormysheva, The chain decomposition of partial permutations and classes of conjugate elements of the semigroup ${\mathcal IS}_n$. Visnyk of Kyiv University, 1993, no.2, 10–18.

\bibitem{GM} O. Ganyushkin and V. Mazorchuk, Classical finite transformation semigroups, Springer, 2009.

\bibitem{Ho} J. Howie, Fundamentals of semigroup theory. London Mathematical Society Monographs, New Series 12. Oxford University Press, 1995.

\bibitem{H1} J. Humphreys, Introduction to Lie algebra and representation theory, Berlin, Heidelberg, New York, 1972.

\bibitem{H2} J. Humphreys, Conjugacy classes in semisimple algebraic groups, Math. Surveys Monographs, 43, Amer. Math. Soc., Providence, RI, 1995.

\bibitem{J} N. Jacobson, Basic algebra II, W. H. Freeman and Company, 1980.

\bibitem{Ka} R. Kane, Reflection groups and invariant theory, Canadian Mathematical Society, 2001.

\bibitem{Ku} G. Kudryavtseva, On conjugacy in regular epigroups, Preprint arXiv:math/0605698.

\bibitem{KM0} G. Kudryavtseva and V. Mazorchuk, On three approaches to conjugacy in semigroups, U.U.D.M. Report 2007:49, Department of Mathematics, Uppsala University, ISSN 1101-3591.

\bibitem{KM1} G. Kudryavtseva and V. Mazorchuk, On conjugation in some transformation and Brauer-type semigroups. Publ. Math. Debrecen 70 (2007), no.1–2, 19–43.

\bibitem{KM2} G. Kudryavtseva and V. Mazorchuk, On the semigroup of square matrices, Preprint arXiv: math/0510624.

\bibitem{La} G. Lallement, Semigroups and combinatorial applications. Pure and Applied Mathematics. A Wiley-Interscience Publication. John Wiley \& Sons, New York-Chichester-Brisbane, 1979.

\bibitem{Zhenheng} Z. Li, The Renner monoids and cell decompositions of classical algebraic monoids, Dissertation, The University of Western Ontario, 2001.

\bibitem{LR} Z. Li and L. Renner, The Renner monoids and cell decompositions of the symplectic algebraic monoids, International Journal of Algebra and Computation, 13(2) 2003, 111-132.


\bibitem{LLC} Z. Li, Z. Li and Y. Cao, Representations of the Renner monoid, International Journal of Algebra and Computation, 19(4) 2009, 511-525.

\bibitem{Lip} S. Lipscomb, Symmetric inverse semigroups, American Mathematical Society, Providence, RI, 1996.

\bibitem{M1} W. Munn, Matrix representations of semigroups, Proc. Camb. Phil. Soc., 53 (1957), 5-12.
\bibitem{M2} W. Munn, The characters of the symmetric inverse semigroup, Proc. Camb. Phil. Soc., 53 (1957), 13-18.

\bibitem{PU1} M. Putcha, Linear Algebraic Monoids, London Math.
Soc. Lecture Note Series 133, Cambridge University Press, 1988.

\bibitem{PU3} M. Putcha, Conjugacy classes in algebraic monoids,
Trans. Amer. Math. Soc. 303 (1987), 529-540.

\bibitem{PU4} M. Putcha, Conjugacy classes in algebraic
monoids II, Canad J. Math. 46 (1994), 648-661.

\bibitem{PR2} M. S. Putcha and L. Renner, The system of idempotents and lattice of $J$-classes of reductive algebraic monoids, J. Algebra 226(1988), 385-399.

\bibitem{R1} L. Renner, Analogue of the Bruhat decomposition
for algebraic monoids, J. of Algebra 101 (1986), 303-338.

\bibitem{R2} L. Renner, Conjugacy classes of semisimple
elements, and irreducible represenations of algebaic monoids,
Comm. Alg 16 (1988), 1933-1943.

\bibitem{R3} L. Renner, Linear Algebraic Monoids, Series:
Encyclopedia of Mathematical Sciences, Springer-Verlag, Vol 134, 2005.


\bibitem{LS1} L. Solomon, An introduction to reductive monoids, Semigroups, Formal Languages and Groups, J. Fountain, ED., Kluwer Academic Publishers, 1995, 295-352.

\bibitem{LS2} L. Solomon, Representations of the rook monoid, J. Algebra 256 (2002), 309-342.


\bibitem{S2} B. Steinberg, M\"{o}bius functions and semigroup representation theory II: Character formulas and multiplicities, Advances in Mathematics, Adv. in Math., 217 (2008), 1521-1557.

\end{thebibliography}
\end{document}